\documentclass[a4paper,12pt]{amsart}
\usepackage{amsmath}
\usepackage{amsthm}
\allowdisplaybreaks[4]
\newtheorem{thm}{Theorem}[section]

\newtheorem{prop}[thm]{Proposition}
\theoremstyle{definition}
\newtheorem{defn}[thm]{Definition}
\theoremstyle{remark}

\numberwithin{equation}{section}

\newcommand{\Real}{\mathbb R}

\newcommand{\grd}[1]{\nabla_{#1}}
\newcommand{\dda}[2]{\nabla_{\!#1}\nabla_{\!#2}}
\newcommand{\hh}[2]{h_{#1m}h_{n#2}g^{mn}}
\newcommand{\rh}[2]{R_{#1m}h_{n#2}g^{mn}}
\newcommand{\rhh}[4]{R_{m#1#2s}h_{#3n}h_{t#4}g^{mn}g^{st}}
\newcommand{\rdh}[4]{R_{m#1#2s}h_{nt}h_{#3#4}g^{mn}g^{st}}
\newcommand{\hhh}[2]{h_{#1k}h_{lm}h_{n#2}g^{kl}g^{mn}}
\newcommand{\hhhh}[4]{h_{#1m}h_{#2s}h_{#3n}h_{#4t}g^{mn}g^{st}}
\newcommand{\hhdh}[4]{h_{#1m}h_{#2s}h_{nt}h_{#3#4}g^{mn}g^{st}}
\newcommand{\dgdt}[2]{\frac{\partial g_{#1#2}}{\partial t}}
\newcommand{\dhdt}[2]{\frac{\partial h_{#1#2}}{\partial t}}

\begin{document}

\title[]{Hyperbolic structures on closed spacelike manifolds}
\author{Kun Zhang}

\address{Department of Mathematics\\Sun Yat-Sen University\\ Guangzhou}
\email{zhangkun@mail2.sysu.edu.cn}%


\date{September 26, 2008}
\begin{abstract}
In this paper, we study the intrinsic mean curvature flow on certain
closed spacelike manifolds, and prove the existence of hyperbolic
structures on them.
\end{abstract}
\maketitle

\section{Introduction}
Recall that a Riemannian manifold $(M,g)$ is hyperbolic if it has
constant negative sectional curvature. These manifolds all come
from the quotient of hyperbolic space $\mathbb{H}^n$ by discrete
isometry groups. However, it is difficult to find a good intrinsic
characterization on the existence of hyperbolic structures on a
given manifold. First, we know that some negatively pinched
Riemannian manifolds can not admit hyperbolic metric. In
\cite{GT}, for $n\geq4$, the counterexample contrasts sharply with
the pinching theorem of positively curved manifolds. In \cite{Fa},
it was shown that for $n\geq10$ the space of negatively curved
metric on some n-manifold is highly non-connected. This implies
that for a given negatively curved metric, it is not always
possible to deform it into a metric with constant negative
curvature by any geometric flows.

\par In this paper, motivated by Lorentzian geometry, we will show that the hyperbolic
structure exists naturally on a large class of spacelike manifolds.
The motivation is the following. It is well known that the imaginary
unit sphere of Minkowski space $\Real^{1,n}$ is the model of
hyperbolic spaces, where under Cartesian coordinates
$(x^0,x^1,\cdots,x^n)$ on $\Real^{1,n}$, the Minkowski metric is
$$g=-(dx^0)^2+(dx^1)^2+\cdots+(dx^n)^2$$ and the equation of imaginary unit sphere is
$$-(x^0)^2+(x^1)^2+\cdots+(x^n)^2=-1.$$ This can be seen from Gauss-Codazzi equations
$$ \left\{
\begin{aligned}
    &R_{ijkl}-(h_{il}h_{jk}-h_{ik}h_{jl})=0\\
    &\nabla_{i}h_{jk}-\nabla_{j}h_{ik}=0
\end{aligned},
\right.
$$
where $h_{ij}$ is the second fundamental form, and $h_{ij}$ equals
to $g_{ij}$ on the imaginary unit sphere. In this paper, we are
interested in an intrinsic generalization of this model.
\begin{defn}
We call a triple $(M,g_{ij},h_{ij})$ a spacelike manifold, if
$(M,g_{ij})$ is a Riemannian manifold, and $h_{ij}$ is a symmetric
tensor satisfying the Gauss-Codazzi equations
$$ \left\{
\begin{aligned}
    &R_{ijkl}-(h_{il}h_{jk}-h_{ik}h_{jl})=0\\
    &\nabla_{i}h_{jk}-\nabla_{j}h_{ik}=0.
\end{aligned}
\right.
$$
\end{defn}
Now we state the main theorem of this paper in the following.
\begin{thm}
Let $(M,g,h)$ be a n-dimensional $(n\geq4)$ closed spacelike
manifold with $h_{ij}>0$, then $M$ admits a hyperbolic metric.
\end{thm}
The idea is to use geometric flows. We define an intrinsic mean
curvature flow of $(g,h)$:
\begin{equation}\label{main}
    \left\{
    \begin{aligned}
        \dgdt{i}{j}\ & =\ -2R_{ij}+2\hh{i}{j}\\
        \dhdt{i}{j}\ & =\ \triangle h_{ij}-\rh{i}{j}-\rh{j}{i}\\
                     &\quad+2\hhh{i}{j}-|A|^2h_{ij}\\
    \end{aligned}
    \right.
\end{equation}
with $g_{ij}(x,0)=\tilde{g}_{ij}(x),\
h_{ij}(x,0)=\tilde{h}_{ij}(x)$, where $\tilde{g}_{ij}(x)$ is the
initial metric on $M$ and $\tilde{h}_{ij}(x)$ is the initial data of
$h_{ij}$ and $|A|^2=g^{ik}g^{jl}h_{ij}h_{kl}$.
\par Mean curvature flow has been intensively studied in recent years
(see \cite{Hu84} for Euclidean ambient space and \cite{Ec} for
Minkowski ambient space). Notice that in extrinsic mean curvature
flow (with ambient space $\Real^{1,n}$), we deform the position
vector $F$ by the evolution equation
\begin{equation*}
    \frac{\partial F}{\partial t}=-H,
\end{equation*}
and \eqref{main} is just the equations of the metric and the second
fundamental form. Here, our observation is that \eqref{main} itself
is also an intrinsicly defined evolution system of $(g,h)$, and it
has its own right to be studied.\\
In this paper, we solve \eqref{main} intrinsicly and show that the
solution exists for all time $[0,\infty)$ and converges (after
normalization)
to a hyperbolic metric.\\[3mm]
\textbf{Acknowledgement} \ I am grateful to my advisor Professor
B.L.Chen for his guidance.
\section{Short-Time Existence and Uniqueness}
Since \eqref{main} is not a strictly parabolic system, in order to
apply theory of strictly parabolic equation to get short time
existence, we use a trick of De Turck by combining our evolution
equation \eqref{main} with the harmonic map flow.
\par Let $(M^n,g_{ij}(x))$ and $(N^m,s_{\alpha\beta}(y))$ be two
Riemannian manifolds, $F:M^n \rightarrow N^m$ be a map. The harmonic
map flow is the following evolution equation for maps from $M^n$ to
$N^m$,
\begin{equation}\label{1}
    \left\{
    \begin{aligned}
        &\frac{\partial}{\partial t}F(x,t)=\triangle F(x,t),\quad
        &&\mathrm {for}\ x \in M^n, t>0, \\
        &F(x,0)=F(x),\quad &&\mathrm {for}\ x\in M^n,
    \end{aligned}
    \right.
\end{equation}
where $\triangle$ is defined by using the metrics $g_{ij}(x)$ and
$s_{\alpha\beta}(y)$ as follows
$$\triangle F^\alpha(x,t)=g^{ij}(x)\grd{i}\grd{j}F^\alpha(x,t),$$
and
\begin{equation}\label{2}
    \grd{i}\grd{j}F^\alpha(x,t)
    =\frac{\partial^2 F^\alpha}{\partial x^i \partial x^j}
    -\Gamma^{k}_{ij}\frac{\partial F^\alpha}{\partial x^k}
    +\tilde{\Gamma}^{\alpha}_{\beta\gamma}\frac{\partial F^\beta}{\partial x^i}
     \frac{\partial F^\gamma}{\partial x^j}.
\end{equation}
Here we use $\{x^i\}$ and $\{y^\alpha\}$ to denote the local
coordinates of $M^n$ and $N^m$ respectively, $\Gamma^{k}_{ij}$ and
$\tilde{\Gamma}^{\alpha}_{\beta\gamma}$ the corresponding
Christoffel symbols of $g_{ij}$ and $s_{\alpha\beta}$. The harmonic
map flow is strictly parabolic, so for any initial data, there
exists a short time smooth solution.
\par Let $(g_{ij}(x,t),h_{ij}(x,t))$ be a complete smooth solution of our
evolution equation \eqref{main}, then the harmonic map flow coupled
with our evolution equation is the following equation:
\begin{equation}\label{3}
    \left\{
    \begin{aligned}
        &\frac{\partial}{\partial t}F(x,t)=\triangle_t F(x,t),\quad
        &&\mathrm {for}\ x \in M^n, t>0, \\
        &F(x,0)=\mathrm {identity},\quad &&\mathrm {for}\ x\in M^n,
    \end{aligned}
    \right.
\end{equation}
where $\triangle_t$ is defined by using the metrics $g_{ij}(x,t)$
and $s_{\alpha\beta}(y)$.
\par Let $(F^{-1})^{*}g$ and $(F^{-1})^{*}h$ be the one-parameter
families of pulled back metrics and pull back tensors on the target
$(N^n,s_{\alpha\beta})$. Denote
$\hat{g}_{\alpha\beta}(y,t)=((F^{-1})^{*}g)_{\alpha\beta}(y,t)$ and
$\hat{h}_{\alpha\beta}(y,t)=((F^{-1})^{*}h)_{\alpha\beta}(y,t)$.
Then by direct calculations, $\hat{g}_{\alpha\beta}(y,t)$ and
$\hat{h}_{\alpha\beta}(y,t)$ satisfy the following evolution
equation:
\begin{equation}\label{4}
    \left\{
    \begin{aligned}
        \frac{\partial\hat{g}_{\alpha\beta}}{\partial t}(y,t)\ & =\
            -2\hat{R}_{\alpha\beta}(y,t)+2\hat{h}_{\alpha\sigma}\hat{h}_{\rho\beta}\hat{g}^{\sigma\rho}
            +\grd{\alpha}V_{\beta}+\grd{\beta}V_{\alpha}\\
        \frac{\partial\hat{h}_{\alpha\beta}}{\partial t}(y,t)\ & =\
            \triangle \hat{h}_{\alpha\beta}(y,t)-\hat{R}_{\alpha\sigma}\hat{h}_{\rho\beta}\hat{g}^{\sigma\rho}
            -\hat{R}_{\beta\sigma}\hat{h}_{\rho\alpha}\hat{g}^{\sigma\rho}\\
            &\quad+2\hat{h}_{\alpha\lambda}\hat{h}_{\mu\nu}\hat{h}_{\rho\beta}\hat{g}^{\lambda\mu}\hat{g}^{\nu\rho}
            -|\hat{A}|^2\hat{h}_{\alpha\beta}\\
            &\quad+\hat{h}_{\beta\gamma}\grd{\alpha}V^{\gamma}
            +\hat{h}_{\alpha\gamma}\grd{\beta}V^{\gamma}\\
    \end{aligned}
    \right.
\end{equation}
where
$V^\alpha=g^{\beta\gamma}(\Gamma^{\alpha}_{\beta\gamma}(\hat{g})
-\tilde{\Gamma}^{\alpha}_{\beta\gamma}(s))$,
$\Gamma^{\alpha}_{\beta\gamma}(\hat{g})$ and
$\tilde{\Gamma}^{\alpha}_{\beta\gamma}(s)$ are the Christoffel
symbols of the metrics $\hat{g}_{\alpha\beta}(y,t)$ and
$s_{\alpha\beta}(y)$ respectively. Here we analysis the principle
part of the right side of \eqref{4}. One can see
\begin{equation*}
    \begin{aligned}
    &-2\hat{R}_{\alpha\beta}(y,t)+2\hat{h}_{\alpha\sigma}\hat{h}_{\rho\beta}\hat{g}^{\sigma\rho}
        +\grd{\alpha}V_{\beta}+\grd{\beta}V_{\alpha}\\
    =&\hat{g}^{\mu\nu}\frac{\partial^2 \hat{g}_{\alpha\beta}}{\partial y^\mu\partial y^\nu}
    +(\mathrm{lower\ order\ terms})
    \end{aligned}
\end{equation*} and
\begin{equation*}
    \begin{aligned}
    &\triangle \hat{h}_{\alpha\beta}(y,t)-\hat{R}_{\alpha\sigma}\hat{h}_{\rho\beta}\hat{g}^{\sigma\rho}
            -\hat{R}_{\beta\sigma}\hat{h}_{\rho\alpha}\hat{g}^{\sigma\rho}\\
    &+2\hat{h}_{\alpha\lambda}\hat{h}_{\mu\nu}\hat{h}_{\rho\beta}\hat{g}^{\lambda\mu}\hat{g}^{\nu\rho}
            -|\hat{A}|^2\hat{h}_{\alpha\beta}
            +\hat{h}_{\beta\gamma}\grd{\alpha}V^{\gamma}
            +\hat{h}_{\alpha\gamma}\grd{\beta}V^{\gamma}\\
    =&\ \hat{g}^{\mu\nu}\Big(\frac{\partial^2 \hat{h}_{\alpha\beta}}{\partial y^\mu\partial y^\nu}
        -\frac{\partial \Gamma^{\sigma}_{\alpha\mu}}{\partial y^\nu}\hat{h}_{\sigma\beta}
        -\frac{\partial \Gamma^{\sigma}_{\beta\mu}}{\partial
        y^\nu}\hat{h}_{\sigma\alpha}\Big)\\
    &-\hat{g}^{\mu\nu}\Big(-\frac{\partial \Gamma^{\sigma}_{\alpha\mu}}{\partial y^\nu}
        +\frac{\partial \Gamma^{\sigma}_{\mu\nu}}{\partial y^\alpha}\Big)\hat{h}_{\sigma\beta}
    -\hat{g}^{\mu\nu}\Big(-\frac{\partial \Gamma^{\sigma}_{\beta\mu}}{\partial y^\nu}
        +\frac{\partial \Gamma^{\sigma}_{\mu\nu}}{\partial
        y^\beta}\Big)\hat{h}_{\sigma\alpha}\\
    &+\hat{g}^{\mu\nu}\frac{\partial \Gamma^{\gamma}_{\mu\nu}}{\partial y^\alpha}\hat{h}_{\gamma\beta}
    +\hat{g}^{\mu\nu}\frac{\partial \Gamma^{\gamma}_{\mu\nu}}{\partial y^\beta}\hat{h}_{\gamma\alpha}
    +(\mathrm{lower\ order\ terms})\\
    =&\ \hat{g}^{\mu\nu}\frac{\partial^2 \hat{h}_{\alpha\beta}}{\partial y^\mu\partial y^\nu}
    +(\mathrm{lower\ order\ terms}).
    \end{aligned}
\end{equation*}
Hence
\begin{equation}\label{5}
    \left\{
    \begin{aligned}
    &\frac{\partial\hat{g}_{\alpha\beta}}{\partial t}(y,t)=
        \hat{g}^{\mu\nu}\frac{\partial^2 \hat{g}_{\alpha\beta}}{\partial y^\mu\partial y^\nu}
        +(\mathrm{lower\ order\ terms})\\
    &\frac{\partial\hat{h}_{\alpha\beta}}{\partial t}(y,t)=
        \hat{g}^{\mu\nu}\frac{\partial^2 \hat{h}_{\alpha\beta}}{\partial y^\mu\partial y^\nu}
        +(\mathrm{lower\ order\ terms})
    \end{aligned}
    \right.
\end{equation}
and we know \eqref{4} is a strictly parabolic system. By theory of
strictly parabolic equations, for any initial data \eqref{4} exists
a smooth short time solution.
\par So we can recover the solution $(g,h)$ for the original
evolution equations from the solution $(\hat{g},\hat{h})$ as
following. Let
$(N^n,s_{\alpha\beta})=(M^n,g_{\alpha\beta}(\cdot,0))$ and since
\begin{equation}\label{18}
V^\alpha=g^{\beta\gamma}(\Gamma^{\alpha}_{\beta\gamma}(\hat{g})
-\tilde{\Gamma}^{\alpha}_{\beta\gamma}(s))=-(\triangle F\circ
F^{-1})^\alpha,
\end{equation}
thus
\begin{equation}\label{19}
    \frac{\partial F}{\partial t}=-V\circ F.
\end{equation}
 Now once having
$\hat{g}_{\alpha\beta}$,we know $V$ and we can solve \eqref{19}
which is just a system of ordinary differential equations on the
domain $M$. Hence $(g,h)$ can be recovered as the pull-back
$g=F^*\hat{g}$ and $h=F^*\hat{h}$.
\par Now we claim the solutions of \eqref{main} with given smooth initial conditions
on a compact manifold are unique. For suppose $(g_1,h_1)$ and
$(g_2,h_2)$ are two solutions which agree at $t=0$. We can solve the
coupled harmonic map flow \eqref{3} for maps $F_1$ and $F_2$ with
the metrics $g_1$ and $g_2$ on $M$ into the same target $N$ with the
same fixed $s$, and starting at the same initial data. Then we have
two solutions $\hat{g}_1$ and $\hat{g}_2$ on $N$ with the same
initial metric. By the standard uniqueness result for strictly
parabolic equations, we have
$(\hat{g}_1,\hat{h}_1)=(\hat{g}_2,\hat{h}_2)$. Hence by \eqref{18}
the corresponding vector fields $V_1=V_2$. Then the solutions of the
two ODE systems
$$\frac{\partial F_1}{\partial t}=-V_1\circ F_1
\quad\mathrm{and}\quad \frac{\partial F_2}{\partial t}=-V_2\circ
F_2$$ with the same initial values must coincide, and hence two
solutions of \eqref{main}
$$(g_1,h_1)=F^*(\hat{g}_1,\hat{h}_1)\quad
and \quad (g_2,h_2)=F^*(\hat{g}_2,\hat{h}_2)$$ must agree.


\section{Preserving Gauss-Codazzi Equations}
In this section, we will show that the Gauss-Codazzi equations are
preserved under \eqref{main}. Let
$G_{ijkl}=R_{ijkl}-(h_{il}h_{jk}-h_{ik}h_{jl})$ and
$C_{ijk}=\nabla_{i}h_{jk}-\nabla_{j}h_{ik}$.
\begin{prop}
If the tensor $h_{ij}$ satisfies Gauss's equation and Codazzi's
equation
$$
\left\{
\begin{aligned}
    &R_{ijkl}-(h_{il}h_{jk}-h_{ik}h_{jl})=0\\
    &\nabla_{i}h_{jk}-\nabla_{j}h_{ik}=0
\end{aligned}
\right.
$$
at time $t=0$, then it remains so for $t>0$.
\end{prop}
\begin{proof}
By direct calculations, we have
$$
    \begin{aligned}
    \frac{\partial}{\partial t}\Gamma_{ij}^{k}
        &=\frac{1}{2}g^{kl}\Big\{
                            \grd{j} \Big(\frac{\partial}{\partial t} g_{il}\Big)
                            +\grd{i} \Big(\frac{\partial}{\partial t} g_{jl}\Big)
                            -\grd{l} \Big(\frac{\partial}{\partial t} g_{ij}\Big)
                          \Big\}\\
    \end{aligned}
$$
    $$\frac{\partial}{\partial t}R_{ijl}^{k}
    =\grd{i}\Big( \frac{\partial}{\partial t} \Gamma_{jl}^{k} \Big)
    -\grd{j}\Big( \frac{\partial}{\partial t} \Gamma_{il}^{k}
            \Big)
        $$
    $$\frac{\partial}{\partial t}R_{ijkl}
        =g_{hk}\frac{\partial}{\partial t}R_{ijl}^{h}+\frac{\partial g_{hk}}{\partial
        t}R_{ijl}^{h}.$$
With these identities we get
\begin{equation*}
    \begin{aligned}
        \frac{\partial}{\partial t}R_{ijkl}
        &=\grd{i}\grd{k}R_{jl}-\grd{i}\grd{l}R_{jk}-\grd{j}\grd{k}R_{il}+\grd{j}\grd{l}R_{ik}\\
        &\quad-\grd{i}\grd{k}(\hh{j}{l})+\grd{i}\grd{l}(\hh{j}{k})\\
        &\quad+\grd{j}\grd{k}(\hh{i}{l})-\grd{j}\grd{l}(\hh{i}{k})\\
        &\quad-R_{ijks}(R_{tl}-\hh{t}{l})g^{st}-R_{ijsl}(R_{tk}-\hh{t}{k})g^{st}\\
    \end{aligned}
\end{equation*}
and the following identity
\begin{equation*}
    \begin{aligned}
        \triangle R_{ijkl}&=-2(B_{ijkl}-B_{ijlk}-B_{iljk}+B_{ikjl})\\
                       &\quad+\grd{i}\grd{k}R_{jl}-\grd{i}\grd{l}R_{jk}-\grd{j}\grd{k}R_{il}+\grd{j}\grd{l}R_{ik}\\
                       &\quad+R_{mjkl}R_{ni}g^{mn}+R_{imkl}R_{nj}g^{mn}
    \end{aligned}
\end{equation*}
where $B_{ijkl}=R_{mijs}R_{nklt}g^{mn}g^{st}$. \\
Then we obtain
\begin{equation}\label{6}
    \begin{aligned}
    &\Big(\frac{\partial}{\partial t}-\triangle \Big)R_{ijkl}
    -2(B_{ijkl}-B_{ijlk}-B_{iljk}+B_{ikjl})\\
    =&\quad-R_{ijks}(R_{tl}-\hh{t}{l})g^{st}-R_{ijsl}(R_{tk}-\hh{t}{k})g^{st}\\
    &\quad-R_{sjkl}(R_{ti}-\hh{t}{i})g^{st}-R_{iskl}(R_{tj}-\hh{t}{j})g^{st}\\
    &\quad-R_{sjkl}\hh{t}{i}g^{st}-R_{iskl}\hh{t}{j}g^{st}\\
    &\quad-\grd{i}\grd{k}(\hh{j}{l})+\grd{i}\grd{l}(\hh{j}{k})\\
    &\quad+\grd{j}\grd{k}(\hh{i}{l})-\grd{j}\grd{l}(\hh{i}{k})\\
    \end{aligned}.
\end{equation}
\par To simplify the evolution equations, we will use a moving frame trick.
 More precisely, let us pick an abstract vector bundle
$V$ over $M$ isomorphic to the tangent bundle $TM$. Choose an
orthonormal frame $F_{a}=F_{a}^{i}\frac{\partial}{\partial x^i},
a=1,\cdots,n$ of $V$ at $t=0$, then evolve $F_{i}^{a}$ by the
equation
$$
    \begin{aligned}
    \frac{\partial}{\partial t}F_{a}^{i}=g^{ij}(R_{jk}-\hh{j}{k})F_{a}^{k}
    \end{aligned}.
$$
Then the frame $F=\{F_{1},\cdots,F_{a},\cdots,F_{n}\}$ will remain
orthonormal for all time. In the following we will use indices
$a,b,\cdots$ on a tensor to denote its components in the evolving
orthonormal frame. In this frame we have the following:
\begin{equation}\label{7}
    \begin{aligned}
    &\Big(\frac{\partial}{\partial t}-\triangle \Big)R_{abcd}
    -2(B_{abcd}-B_{abdc}-B_{adcb}+B_{acbd})\\
    =&\quad-R_{sbcd}\hh{t}{a}g^{st}-R_{ascd}\hh{t}{b}g^{st}\\
    &\quad-\grd{a}\grd{c}(\hh{b}{d})+\grd{a}\grd{d}(\hh{b}{c})\\
    &\quad+\grd{b}\grd{c}(\hh{a}{d})-\grd{b}\grd{d}(\hh{a}{c})\\
    \end{aligned}
\end{equation}
and
\begin{equation}\label{8}
    \Big(\frac{\partial}{\partial t}-\triangle\Big)h_{ab}=-|A|^2h_{ab}.
\end{equation}
By calculations, we have
\begin{equation}\label{9}
    \begin{aligned}
    &\Big(\frac{\partial}{\partial t}-\triangle \Big)
    \{R_{abcd}-(h_{ad}h_{bc}-h_{ac}h_{bd})\}\\
    =\ &2(B_{abcd}-B_{abdc}-B_{adcb}+B_{acbd})\\
    &-R_{sbcd}\hh{t}{a}g^{st}-R_{ascd}\hh{t}{b}g^{st}\\
    &-\grd{a}\grd{c}(\hh{b}{d})+\grd{a}\grd{d}(\hh{b}{c})\\
    &+\grd{b}\grd{c}(\hh{a}{d})-\grd{b}\grd{d}(\hh{a}{c})\\
    &+2|A|^{2}(h_{ad}h_{bc}-h_{ac}h_{bd})\\
    &+2(\grd{m}h_{ad}\grd{n}h_{bc}-\grd{m}h_{ac}\grd{n}h_{bd})g^{mn}.
    \end{aligned}
\end{equation}
Then we want to replace $B_{abcd}$ by
$$\tilde{B}_{abcd}=\{R_{mabs}-(h_{ms}h_{ab}-h_{mb}h_{as})\}
\{R_{mcds}-(h_{ms}h_{cd}-h_{md}h_{cs})\}g^{mn}g^{st}$$ and replace
terms including $\nabla h,\nabla\nabla h$ by $C$ and $\nabla C$
respectively.\\
That is
\begin{equation}\label{10}
    \begin{aligned}
    &\quad B_{abcd}-B_{abdc}-B_{adcb}+B_{acbd}\\
   =&\ \tilde{B}_{abcd}-\tilde{B}_{abdc}-\tilde{B}_{adcb}+\tilde{B}_{acbd}\\
   &-\rhh{a}{b}{d}{c}-\rhh{c}{d}{b}{a}\\
   &+\rhh{a}{b}{c}{d}+\rhh{d}{c}{b}{a}\\
   &-\rdh{a}{d}{c}{b}+\rhh{a}{d}{c}{b}\\
   &-\rdh{b}{c}{a}{d}+\rhh{b}{c}{d}{a}\\
   &+\rdh{a}{c}{d}{b}-\rhh{a}{c}{d}{b}\\
   &+\rdh{b}{d}{a}{c}-\rhh{b}{d}{c}{a}\\
   &-\hhhh{a}{b}{c}{d}+\hhhh{a}{b}{d}{c}\\
   &+h_{ad}h_{bc}|A|^2-\hhdh{a}{d}{b}{c}-\hhdh{b}{c}{a}{d}\\
   &-h_{ac}h_{bd}|A|^2+\hhdh{a}{c}{b}{d}+\hhdh{b}{d}{a}{c}
    \end{aligned}
\end{equation}
and
\begin{equation}\label{11}
    \begin{aligned}
    &-\dda{a}{c}(\hh{b}{d})+\dda{a}{d}(\hh{b}{c})+\dda{b}{c}(\hh{a}{d})\\
    &-\dda{b}{d}(\hh{a}{c})
    +2(\grd{m}h_{ad}\grd{n}h_{bc}g^{mn}-\grd{m}h_{ac}\grd{n}h_{bd}g^{mn})\\
       \end{aligned}
\end{equation}
$$
\begin{aligned}
 =&-\nabla_{c}(\nabla_{a}h_{bm}-\nabla_{b}h_{am})h_{nd}g^{mn}-\nabla_{a}(\nabla_{c}h_{dm}-\nabla_{d}h_{cm})h_{nb}g^{mn}\\
    &+\nabla_{d}(\nabla_{a}h_{bm}-\nabla_{b}h_{am})h_{nc}g^{mn}-\nabla_{b}(\nabla_{c}h_{dm}-\nabla_{d}h_{cm})h_{na}g^{mn}\\
    &-(\nabla_{a}h_{bm}-\nabla_{b}h_{am})(\nabla_{c}h_{dn}-\nabla_{d}h_{cn})g^{mn}\\
    &-(\nabla_{a}h_{dm}-\nabla_{m}h_{ad})\nabla_{c}h_{bn}g^{mn}-(\nabla_{d}h_{am}-\nabla_{m}h_{ad})\nabla_{b}h_{cn}g^{mn}\\
    &+(\nabla_{a}h_{cm}-\nabla_{m}h_{ac})\nabla_{d}h_{bn}g^{mn}+(\nabla_{c}h_{am}-\nabla_{m}h_{ac})\nabla_{b}h_{dn}g^{mn}\\
    &+(\nabla_{m}h_{bc}-\nabla_{c}h_{mb})\nabla_{n}h_{ad}g^{mn}+(\nabla_{m}h_{bc}-\nabla_{b}h_{mc})\nabla_{n}h_{ad}g^{mn}\\
    &-(\nabla_{m}h_{bd}-\nabla_{d}h_{mb})\nabla_{n}h_{ac}g^{mn}-(\nabla_{m}h_{bd}-\nabla_{b}h_{md})\nabla_{n}h_{ac}g^{mn}\\
    &-R_{acbm}h_{ns}h_{td}g^{mn}g^{st}-R_{acms}h_{nd}h_{tb}g^{mn}g^{st}+R_{bcam}h_{ns}h_{td}g^{mn}g^{st}\\
    &+R_{bcms}h_{nd}h_{ta}g^{mn}g^{st}+R_{adbm}h_{ns}h_{tc}g^{mn}g^{st}+R_{adms}h_{nc}h_{tb}g^{mn}g^{st}\\
    &-R_{bdam}h_{ns}h_{tc}g^{mn}g^{st}-R_{bdms}h_{nc}h_{ta}g^{mn}g^{st}.
\end{aligned}$$
Let us denote curvature tensor by $Rm$ and denote any tensor product
of two tensors $S$ and $T$ by $S*T$ when we do not need the precise
expression. Therefore, if we replace terms including $Rm*h*h$ by
term $G*h*h$, with \eqref{9}\eqref{10}\eqref{11} and by some
calculation we obtain
\begin{equation}\label{12}
    \begin{aligned}
    \Big(\frac{\partial}{\partial t}-\triangle\Big)G=G*G+G*h*h+\nabla C*h+C*\nabla h+C*C
    \end{aligned},
\end{equation}
where $G_{ijkl}=R_{ijkl}-(h_{il}h_{jk}-h_{ik}h_{jl})$ and
$C_{ijk}=\nabla_{i}h_{jk}-\nabla_{j}h_{ik}$.\\
Since we have
\begin{equation*}
    \begin{aligned}
    &\frac{\partial}{\partial t}\nabla_{i}h_{jk}
    =\nabla_{i}\Big(\frac{\partial}{\partial t}h_{jk}\Big)
        -\Big(\frac{\partial}{\partial t}\Gamma_{ij}^{l}\Big)h_{lk}-\Big(\frac{\partial}{\partial
        t}\Gamma_{ik}^{l}\Big)h_{lj}\\
    =&\nabla_{i}(\triangle h_{jk}-R_{jm}h_{nk}g^{mn}-R_{km}h_{nj}g^{mn}+2h_{jm}h_{ns}h_{tk}g^{mn}g^{st}-|A|^{2}h_{jk})\\
    &-\Big(\frac{\partial}{\partial t}\Gamma_{ij}^{l}\Big)h_{lk}+\grd{i}\rh{k}{j}+\grd{k}\rh{i}{j}-\grd{m}R_{ik}h_{nj}g^{mn}\\
    &-\grd{i}h_{km}h_{ns}h_{tj}g^{mn}g^{st}-\grd{i}h_{ms}h_{nk}h_{tj}g^{mn}g^{st}-\grd{k}h_{im}h_{ns}h_{tj}g^{mn}g^{st}\\
    &-\grd{k}h_{ms}h_{ni}h_{tj}g^{mn}g^{st}+\grd{m}h_{is}h_{nj}h_{tk}g^{mn}g^{st}+\grd{m}h_{ks}h_{nj}h_{ti}g^{mn}g^{st}
    \end{aligned}
\end{equation*}
and
\begin{equation*}
    \begin{aligned}
    &\triangle(\grd{i}h_{jk})
    =g^{mn}\dda{m}{n}(\grd{i}h_{jk})\\
    =&\grd{i}(\triangle h_{jk})+R_{im}\grd{n}h_{jk}g^{mn}+2(R_{mijs}\grd{n}h_{tk}+R_{miks}\grd{n}h_{tj})g^{mn}g^{st}\\
    &+\grd{j}\rh{i}{k}-\grd{m}R_{ij}h_{nk}g^{mn}+\grd{k}\rh{i}{j}-\grd{m}R_{ik}h_{nj}g^{mn}.
    \end{aligned}
\end{equation*}
So we get
\begin{equation}\label{13}
    \begin{aligned}
    &\Big(\frac{\partial}{\partial t}-\triangle\Big)\grd{i}h_{jk}+
    \Big(\frac{\partial}{\partial t}\Gamma_{ij}^{l}\Big)h_{lk}\\
    =&-R_{jm}\grd{i}h_{nk}g^{mn}-R_{km}\grd{i}h_{nj}g^{mn}\\
    &-R_{im}\grd{n}h_{jk}g^{mn}
    +\grd{i}(2h_{jm}h_{ns}h_{tk}g^{mn}g^{st}-|A|^2h_{jk})\\
    &-\grd{i}\rh{j}{k}-\grd{j}\rh{i}{k}\\
    &-2(R_{mijs}\grd{n}h_{tk}+R_{miks}\grd{n}h_{tj})g^{mn}g^{st}\\
    &-\grd{i}h_{km}h_{ns}h_{tj}g^{mn}g^{st}-\grd{i}h_{ms}h_{nk}h_{tj}g^{mn}g^{st}\\
    &-\grd{k}h_{im}h_{ns}h_{tj}g^{mn}g^{st}-\grd{k}h_{ms}h_{ni}h_{tj}g^{mn}g^{st}\\
    &+\grd{m}h_{is}h_{tk}h_{nj}g^{mn}g^{st}+\grd{m}h_{ks}h_{ti}h_{nj}g^{mn}g^{st}.
    \end{aligned}
\end{equation}
Then in the moving frame we obtain
\begin{equation}\label{14}
    \begin{aligned}
    &\Big(\frac{\partial}{\partial t}-\triangle\Big)\grd{a}h_{bc}+|A|^{2}\grd{a}h_{bc}
    +\Big(\frac{\partial}{\partial
        t}\Gamma_{ij}^{l}\Big)h_{lk}F_{a}^{i}F_{b}^{j}F_{c}^{k}\\
    =&-\grd{a}h_{cm}h_{ns}h_{tb}g^{mn}g^{st}-\grd{a}h_{mb}h_{ns}h_{tc}g^{mn}g^{st}\\
    &-\grd{m}h_{bc}h_{ns}h_{ta}g^{mn}g^{st}
    +2\grd{a}h_{bm}h_{ns}h_{tkc}g^{mn}g^{st}\\
    &+2\grd{a}h_{cm}h_{ns}h_{tb}g^{mn}g^{st}+2\grd{a}h_{ms}h_{nb}h_{tc}g^{mn}g^{st}\\
    &-2\grd{a}h_{ms}h_{nt}h_{bc}g^{mn}g^{st}-\grd{a}h_{cm}h_{ns}h_{tb}g^{mn}g^{st}\\
    &-\grd{a}h_{ms}h_{nb}h_{tc}g^{mn}g^{st}
    -\grd{c}h_{am}h_{ns}h_{tb}g^{mn}g^{st}\\
    &+\grd{m}h_{as}h_{nb}h_{tc}g^{mn}g^{st}+\grd{m}h_{cs}h_{nb}h_{ta}g^{mn}g^{st}\\
    &-2R_{mabs}\grd{n}h_{tc}g^{mn}g^{st}-2R_{macs}\grd{n}h_{tb}g^{mn}g^{st}.\\
    \end{aligned}
\end{equation}
Then we replace terms including $\grd{}h$ by $C$ and terms including
$Rm$ by $G$. Finally, we have
\begin{equation}\label{15}
    \Big(\frac{\partial}{\partial
    t}-\triangle\Big)C=-|A|^2C+C*h*h+C*Rm+G*\grd{}h
\end{equation}
Combing \eqref{12}\eqref{15}, we obtain
\begin{equation}\label{16}
    \begin{aligned}
    &\Big(\frac{\partial}{\partial t}-\triangle\Big)(|G|^2+|C|^2)\\
    \leq &C_1(|G|^2+|C|^2)-2|\grd{}G|^2-2|\grd{}C|^2\\
    &+\langle G, G*G+G*h*h+\nabla C*h+C*\nabla h+C*C \rangle\\
    &+\langle C, -|A|^2C+C*h*h+C*Rm+G*\grd{}h\rangle\\
    \leq &C_2(|G|^2+|C|^2)
    \end{aligned}
\end{equation}
where we use Cauchy-Schwarz inequality, and for $0\leq t<\delta$ we
have bounded $|Rm|,|A|,|\nabla h|$. Thus, by the standard maximum
principle
\begin{equation*}
    \frac{d}{dt}(|G|^2+|C|^2)_{max} \leq C_2(|G|^2+|C|^2)_{max},
\end{equation*}
we get
$$(|G|^2+|C|^2)_{max}(t)\leq e^{C_2t}(|G|^2+|C|^2)_{max}(0).$$
Since $(|G|^2+|C|^2)_{max}(0)=0,$ the Gauss-Codazzi equations are
preserved as long as the solution exists.
\end{proof}
In the following we will still call $h_{ij}(x,t)$ the second
fundamental form and its trace $H$ the mean curvature.


\section{Evolution of metric and curvature}
Using Gauss-Codazzi equations, we rewrite our evolution equations in
the following
\begin{prop}
\begin{subequations}\label{pro}
    \begin{align}
    &\frac{\partial}{\partial t}g_{ij}=2Hh_{ij}\label{pro}\\
    &\Big(\frac{\partial}{\partial t}-\triangle\Big)h_{ij}=2H\hh{i}{j}-|A|^2h_{ij}\\
    &\Big(\frac{\partial}{\partial t}-\triangle\Big)H =-H|A|^2\\
    &\Big(\frac{\partial}{\partial t}-\triangle\Big)|A|^2
    =-2|\grd{}A|^2-2|A|^4.
    \end{align}
\end{subequations}
\end{prop}
\par Since $h_{ij}$ is positive at $t=0$ and $M$ is compact,
there are some $\varepsilon >0\ and\ \beta
>0$,such that $\beta Hg_{ij} \geq h_{ij} \geq \varepsilon Hg_{ij}$
at $t=0$ holds on $M$. We want to show that inequality remains true
as long as the solution of our evolution equation \eqref{main}
exists. For this purpose we need the following maximum principle for
tensor on manifolds, which is proved in \cite{Ha82}.\par Let $u^{k}$
be a vector field and let $M_{ij}$ and $N_{ij}$ be symmetric tensors
on a compact manifold $M$ which may all depend on time $t$. Assume
that $N_{ij}=p(M_{ij},g_{ij})$ is a polynomial in $M_{ij}$ formed by
contracting products of $M_{ij}$ with itself using the metric.
Furthermore, let this polynomial satisfy a null-eigenvector
condition, i.e. for any null-eigenvector $X$ of $M_{ij}$ we have
$N_{ij}X^{i}X^{j} \geq 0$. Then we have
\begin{thm}[Hamilton]\label{thm1}
Suppose that on $0 \leq t < T$ the evolution equation
\begin{equation*}
\frac{\partial}{\partial t}M_{ij}=\triangle
M_{ij}+u^k\grd{k}M_{ij}+N_{ij}
\end{equation*}
holds, where $N_{ij}=p(M_{ij},g_{ij})$ satisfies the
null-eigenvector condition above. If $M_{ij} \geq 0\ at\ t=0$, then
it remains so on $0 \leq t <T$.
\end{thm}\par
An immediate consequence is
\begin{prop}
If $\varepsilon Hg_{ij} \leq h_{ij} \leq \beta Hg_{ij}$, and $H>0$
at $t=0$, then these remain so as long as the solution of
\eqref{main} exists.
\end{prop}
\begin{proof}
First, by using maximum principle on
$$\Big(\frac{\partial}{\partial t}-\triangle\Big)H =-H|A|^2,$$ we
know $H>0$ as long as the solution of \eqref{main} exists.\\
Then we consider
$$
\begin{aligned}
    M_{ij}&=h_{ij}-\varepsilon Hg_{ij}\\
    \frac{\partial M_{ij}}{\partial t}&=\frac{\partial h_{ij}}{\partial  t}
        -\varepsilon \frac{\partial H}{\partial t}g_{ij}
        -\varepsilon H\frac{\partial g_{ij}}{\partial t}\\
    &=\triangle h_{ij}+2H\hh{i}{j}-|A|^2h_{ij}\\
    &\quad-\varepsilon(\triangle H-|A|^2H)g_{ij}-\varepsilon H(2Hh_{ij})\\
    &=\triangle M_{ij}+2H\hh{i}{j}\\
    &\quad-|A|^2(h_{ij}-\varepsilon Hg_{ij})-2\varepsilon H^2h_{ij}
\end{aligned}
$$
For any null vector $v^i$ of $M_{ij}$, we have
$$
    \begin{aligned}
    &[2H\hh{i}{j}-|A|^2(h_{ij}-\varepsilon Hg_{ij})-2\varepsilon
    H^2h_{ij}]v^j\\
    =&\ 2Hh_{im}g^{mn}(\varepsilon Hv_{n})-2\varepsilon H^2(\varepsilon Hv_i)\\
    =&\ 2H(\varepsilon Hv_i)\varepsilon H-2\varepsilon H^2(\varepsilon Hv_i)\\
    =&\ 0
    \end{aligned}
$$
Thus, $\varepsilon Hg_{ij} \leq h_{ij}$ follows from theorem
\ref{thm1}. Then $h_{ij} \leq \beta Hg_{ij}$ follows in the same
way.
\end{proof}

Finally, we state the higher derivative estimate in the following
proposition.
\begin{prop}\label{de}
There exist constants $C_m ,m=1,2,\cdots,$ such that if the second
fundamental form of a complete solution to our evolution equation is
bounded by $$|A| \leq M$$ up to time t with $0<t\leq 1/M$ , then the
covariant derivative of the second fundamental form is bounded by
$$|\grd{}A|\leq C_{1}M/\sqrt{t}$$ and the $m^{th}$ covariant
derivative of the second fundamental form is bounded by
$$|\grd{}^{m}A|\leq C_{m}M/t^{\frac{m}{2}}.$$
Here the norms are taken with respect to the evolving metric.
\end{prop}
\begin{proof}
By direct caculation, for any m we have an equation
    \begin{equation*}
    \Big(\frac{\partial}{\partial
    t}-\triangle\Big)|\grd{}^{m}A|^2=-2|\grd{}^{m+1}A|^2+\sum_{i+j+k=m}
        \grd{}^{i}A*\grd{}^{j}A*\grd{}^{k}A*\grd{}^{m}A .
    \end{equation*}
So we can follow the same way using a somewhat standard Bernstein
estimate in PDEs to get our theorem(see \cite{Sh89} for Ricci flow).
\end{proof}

\section{Monotonicity formula and Long time behaviors}
First, by positivity of $h_{ij}$ we have $$H^2/n\leq|A|^2 < H^2.$$
Then from (4.1c) we get
$$-H^3<\Big(\frac{\partial}{\partial t}-\triangle\Big)H \leq-\frac{H^3}{n}.$$
Thus by maximum principle we obtain
\begin{equation}\label{5.1}
\frac{1}{\sqrt{2t+\frac{1}{H^{2}_{min}(0)}}} < H(t)
\leq\frac{1}{\sqrt{\frac{2}{n}t+\frac{1}{H^{2}_{max}(0)}}}.
\end{equation}
With applying maximum principle on (4.1d) again, we have
\begin{equation*}
|A|^2(t) \leq\frac{1}{2t+\frac{1}{|A|^{2}_{max}(0)}}.
\end{equation*}
Since
$$
\frac{1}{2nt+\frac{n}{H^{2}_{min}(0)}}<H^2(t)/n\leq|A|^2(t),
$$
we get
\begin{equation}\label{5.2}
\frac{1}{2nt+\frac{n}{H^{2}_{min}(0)}}< |A|^2(t)
\leq\frac{1}{2t+\frac{1}{|A|^{2}_{max}(0)}}.
\end{equation}
In particular, \eqref{5.2} implies $$|A|\rightarrow 0 \quad as \quad
t \rightarrow + \infty.$$ Combining with our derivatives estimate
(Proposition 4.4) we know the solution of our evolution equation
\eqref{main} exists for all the time.
\par We need the
following monotonicity formula to understand the long time behaviors
of the solution to \eqref{main}.
\begin{prop}
If $(g_{ij}(t),h_{ij}(t))$ is the solution of \eqref{main}, then we
have the formula
$$\frac{\partial}{\partial t}\int_{M}H^{n}d\mu_{t}
=-n(n-1)\int_{M}\frac{|\grd{}H|^2}{H^2}H^{n}d\mu_{t}
-n\int_{M}|h_{ij}-\frac{1}{n}Hg_{ij}|^{2}H^{n}d\mu_{t}.$$
\end{prop}
\begin{proof}
It follows from the evolution equations of Proposition 4.1 and
direct calculation.
\end{proof}
From proposition 5.1 we know
\begin{equation}\label{5.3}
0<\int_{M}H^{n}d\mu_{t}< C
\end{equation}
for all $t\in[0,+\infty)$.\\
This implies
\begin{equation*}
    \left\{
    \begin{aligned}
    &\int_{0}^{\infty}\!\!\!\!\int_{M}\frac{|\grd{}H|^2}{H^2}H^{n}d\mu_{t}<\infty\\
    &\int_{0}^{\infty}\!\!\!\!\int_{M}|h_{ij}-\frac{1}{n}Hg_{ij}|^{2}H^{n}d\mu_{t}<\infty
    \end{aligned}
    \right..
\end{equation*}
In particular, there is a sequence $t_k\rightarrow + \infty$ such
that
\begin{equation}\label{6.1}
t_k\int_{M}\frac{|\grd{}H|^2}{H^2}H^{n}d\mu_{t_k}\rightarrow0 \quad
as \quad k\rightarrow\infty
\end{equation}
and
\begin{equation}\label{6.2}
t_k\int_{M}|h_{ij}-\frac{1}{n}Hg_{ij}|^{2}H^{n}d\mu_{t_k}\rightarrow0
\quad as \quad k\rightarrow\infty.
\end{equation}
\par
Denote by
$$\epsilon_k=\frac{1}{|A|_{max}(t_k)} .$$
We parabolically scale the solution and shift the time $t_k$ to the
origin 0,
$$
\begin{aligned}
\tilde{g}_{ij}^k(\cdot,\tilde{t})=\epsilon_{k}^{-2}g_{ij}(\cdot,t_k+\epsilon_{k}^{2}\tilde{t}),\\
\tilde{h}_{ij}^k(\cdot,\tilde{t})=\epsilon_{k}^{-1}h_{ij}(\cdot,t_k+\epsilon_{k}^{2}\tilde{t}),\\
where \quad\tilde{t}\in[-t_k/\epsilon_{k}^{2},+\infty).
\end{aligned}
$$
We can check that
$(\tilde{g}_{ij}^k(\cdot,\tilde{t}),\tilde{h}_{ij}^k(\cdot,\tilde{t}))$
is still a solution to \eqref{main}. \\
Since
$$|\tilde{A}^k(\cdot,\tilde{t})|^2=
\frac{|A(\cdot,t_k+\epsilon_{k}^{2}\tilde{t})|^2}{|A|_{max}^2(t_k)},$$
and \eqref{5.2}, it follows that
\begin{equation}
\frac{1}{C_1}<|\tilde{A}^k(\cdot,\tilde{t})|^2<C_1 \quad
\mathrm{for}\ \tilde{t}\in[-t_k/2\epsilon_{k}^{2},0],
\end{equation}
where the constant $C_1$ is independent of $k$.
\par By our derivatives
estimate (Proposition 4.4), the uniform bound of the second
fundamental form $|\tilde{A}^k(\cdot,\tilde{t})|$ implies the
uniform bound on all the derivatives of the second fundamental form
at $\tilde{t}=0$ for all $k$. By Gauss equation we have uniform
bound of the curvature and all the derivatives of the curvature at
$\tilde{t}=0$ for all $k$.\\
By \eqref{5.3} we know
$$\int_{M}(\tilde{H}^k(\cdot,0))^{n}d\tilde{\mu}_{0}<C_2.$$
Combining with \eqref{5.1} it follows
\begin{equation}\label{5.5}
\mathrm{Vol}(M,\tilde{g}_{ij}^{k}(\cdot,0))<C_3.
\end{equation}
On the other hand, by Proposition 4.3, \eqref{5.2} and Gauss
equation we have
\begin{equation}\label{5.4}
0>-\frac{1}{C_4}\geq sec(M,\tilde{g}_{ij}^{k}(\cdot,0))>-1.
\end{equation}
With \eqref{5.4} and \eqref{5.5}, we can get the uniform upper bound
on their diameters and uniform lower bound on their volumes by using
the following theorem .
\begin{thm}[Gromov\cite{Gr}]
Let M be an n-dimensional closed Riemannian manifold of negative
curvature and $Sec(M) \geq-1$. If $n\geq8$, then $Vol(M)\geq
C(1+d(M))$ and for n=4,5,6,7, $Vol(M)\geq C(1+d^{1/3}(M))$, where we
denote volume of M by Vol(M), diameter of M by d(M) and the constant
$C>0$ depends only on n.
\end{thm}
Now we know
$(M,\tilde{g}_{ij}^{k}(\cdot,0),\tilde{h}_{ij}^{k}(\cdot,0))$ is a
sequence which have uniform bound on sectional curvature, uniform
upper bound on diameters and uniform lower bound on volumes. Using
cheeger's Lemma in \cite{CE} we have the uniform lower bound of
their injective radii with respect to $\tilde{g}_{ij}^{k}(\cdot,0)$
for $n\geq 4$. Then we can apply the same argument of Hamilton's
compactness theorem in \cite{Ha95} to extract a convergent
subsequence
$(M,\tilde{g}_{ij}^{k_l}(\cdot,0),\tilde{h}_{ij}^{k_l}(\cdot,0))$
from $(M,\tilde{g}_{ij}^{k}(\cdot,0),\tilde{h}_{ij}^{k}(\cdot,0))$.
More precisely, there is a triple
$(M_{\infty},\tilde{g}_{ij}^{\infty}(\cdot,0),\tilde{h}_{ij}^{\infty}(\cdot,0))$
and a sequence of diffeomorphisms $f_l:M_{\infty}\rightarrow M_l$.
Notice that $M_{\infty}$ is diffeomorphism to $M$, since we have
uniform diameter bound.\! And the pull-back metrics
$(f_l)^{*}\tilde{g}_{ij}^{k_l}(\cdot,0)$ and the pull-back second
fundamental forms $(f_l)^{*}\tilde{h}_{ij}^{k_l}(\cdot,0)$ converge
in $C^\infty$ topology to
$(\tilde{g}_{ij}^{\infty}(\cdot,0),\tilde{h}_{ij}^{\infty}(\cdot,0))$
.
\par
From \eqref{6.1} and \eqref{6.2} we obtain
\begin{equation*}
t_{k_l}\epsilon_{k_l}^{-2}\int_{M}\frac{|\tilde{\grd{}}
\tilde{H}^{k_l}|^2(0)}{(\tilde{H}^{k_l})^2(0)}(\tilde{H}^{k_l})^{n}(0)d\tilde{\mu}_{t_{k_l}}\rightarrow0
\quad as \quad l\rightarrow\infty
\end{equation*}
and
\begin{equation*}
t_{k_l}\epsilon_{k_l}^{-2}\int_{M}|\tilde{h}^{k_l}_{ij}
-\frac{1}{n}\tilde{H}^{k_l}\tilde{g}^{k_l}_{ij}|^{2}(0)(\tilde{H}^{k_l})^{n}(0)d\tilde{\mu}_{t_{k_l}}\rightarrow0
\quad as \quad l\rightarrow\infty.
\end{equation*} Here the norm is taken with respect to
$\tilde{g}^{k_l}_{ij}(0)$.\\
Notice that $t_{k_l}\epsilon_{k_l}^{-2}$ and $|\tilde{H}^{k_l}(0)|$
and $\mathrm{Vol}(M,\tilde{g}_{ij}^{k_l}(\cdot,0))$ have uniform
lower bound, we have
$$|\tilde{\grd{}}\tilde{H}^{k_l}|(0)\rightarrow0\quad as \quad l\rightarrow\infty$$ and
$$|\tilde{h}_{ij}^{k_l}-\frac{1}{n}\tilde{H}^{k_l}\tilde{g}^{k_l}_{ij}|(0)\rightarrow0\quad as \quad l\rightarrow\infty.$$
Therefore, by Gauss equation, we know the sectional curvature of
$(M_{\infty},\tilde{g}_{ij}^{\infty}(\cdot,0),\tilde{h}_{ij}^{\infty}(\cdot,0))$
is a constant($\equiv-1/n$).


\bibliographystyle{amsplain}

\end{document}